\documentclass[12pt]{amsart}

\newtheorem{theorem}{Theorem}[section]

\newtheorem{corollary}[theorem]{Corollary}

\newtheorem{lemma}[theorem]{Lemma}

\theoremstyle{definition}

\newtheorem{remark}[theorem]{Remark}

\newtheorem{example}[theorem]{Example}

\theoremstyle{parrafo}

\begin{document}

\title[]{A measure-theoretic version of the Dragomir-Jensen inequality}

\author{J. M. Aldaz}
\address{Departamento de Matem\'aticas,
Universidad  Aut\'onoma de Madrid, Cantoblanco 28049, Madrid, Spain.}
\email{jesus.munarriz@uam.es}


\thanks{Partially supported by Grant  MTM2009-12740-C03-03 of the
D.G.I. of Spain}






\begin{abstract} We extend Dragomir's refinement of Jensen's inequality
from the dicrete to the general case, identifying the equality conditions. 
\end{abstract}


\maketitle


\markboth{J. M. Aldaz}{Dragomir-Jensen}

\section{Introduction} Suppose we have a probability space $(\Omega, \mathcal{A}, P)$, an integrable function
 $X:\Omega\to\mathbb{R}$, and a real valued strictly convex function 
 $\phi$ defined in some  interval 
containing the range of $X$. By Jensen's inequality,
\begin{equation}\label{spread}
E \phi (X)  -  \phi \left(E X \right) \ge 0,
\end{equation}
with equality precisely when $X(\omega)\equiv E( X )$ for $P$-a.e. $\omega$. Thus, the left hand side of (\ref{spread}) provides a measure of the spread of $X$
around its mean value. Of course, in the important special case
 $\phi(t) = t^2$, 
the left hand side of (\ref{spread})
is just the variance of $X$. It is natural to study how these
generalized variances change when the probability $P$ varies.
The case of  discrete 
probability measures with finite support was considered by S. S. Dragomir
in  \cite{Dra}. 

Here we  note
that Dragomir's clever proof (refining Jensen's inequality by using Jensen's inequality itself) 
can be used to extend his result to general 
probability spaces. 
Additionally, we identify the
cases of equality when $\phi$ is strictly convex, 
and present some immediate consecuences.

\section{The Dragomir-Jensen Inequality}

To motivate the inequality,  consider the following
simple example: Let $X$ and $Y$ be   random variables
on $(\Omega,\mathcal{A}, P)$, with densities
$f_X$ and   $f_Y$ respectively, such that 
the
quotient $f_Y/f_X$ is well defined almost everywhere, 
or in other words, such that
the distribution of Y, defined by 
the push-forward (or induced) probability $Y_*P (A) := P(Y\in A)$, 
is absolutely continuous with respect to $X_* P(A):= P(X\in A)$. If both $X$ and
$Y$ have mean zero, then it is easy to bound
 the variance
of $Y$ in terms of the variance of $X$: Since
\begin{equation}\label{variance}
\operatorname{Var} (Y) 
= \int_{-\infty}^\infty   y^2   f_Y(y) d y 
= \int_{-\infty}^\infty  y^2  \ \frac{f_Y(y)}{f_X(y)} \ f_X(y) d y,
\end{equation}
and $\operatorname{Var} (X) 
= \int_{-\infty}^\infty   y^2   f_X(y) d y $, replacing $\frac{f_Y(y)}{f_X(y)}$ in (\ref{variance}) by its  essential supremum and by its essential infimum,
we get
\begin{equation}\label{variance2}
\operatorname{ess\ inf} \left(\frac{f_Y(y)}{f_X(y)}\right) \operatorname{Var} (X)
\le \operatorname{Var} (Y) 
\le 
\left\|\frac{f_Y(y)}{f_X(y)}\right\|_\infty \operatorname{Var} (X).
\end{equation}
In fact, (\ref{variance2}) holds whenever the expected values of $X$ and $Y$ are finite,
and not just for $\phi(t) = t^2$, but for arbitrary convex functions. This is the content of the (two sided) Dragomir-Jensen inequality (cf. Theorem \ref{DraJen} and Corollary \ref{DraJen2}), which generalizes (\ref{variance2})
much in the same way as Jensen's inequality generalizes 
$\operatorname{Var} (X) \ge 0$. 

Next we recall some basic notations
and facts.
Given 
an absolutely continuous   measure $P << Q$,
as usual, $\frac{d P}{d Q}$ 
denotes the  Radon-Nikodym derivative of $P$ with respect to $Q$,
and $\left\|\frac{d P}{d Q}\right\|_\infty$ denotes its essential supremum
(which could be infinite). We adopt the standard measure-theoretic convention
$\infty \cdot 0 = 0$ (recall that under any other convention, 
the monotone convergence theorem
would fail). Of course, integrals are used here in the Lebesgue sense. In particular, to
have $\int X dP$ well defined, it is assumed that either
$\int X^+ dP < \infty$ or $\int X^- dP < \infty$, where
$X^+$ and $X^-$ respectively denote the positive and negative parts of $X$. Thus, we do not consider principal values. 

Since we will be dealing with positive measures, 
and in fact, probabilities,
by taking any representative of the Radon-Nikodym derivative
and
redefining it (if needed)  on a set of
measure zero, we may assume 
that
$0\le \frac{d P}{d Q} < \infty$ (we use the same notation for 
the Radon-Nikodym derivative and its representative).

\begin{theorem}\label{DraJen} Let $(\Omega, \mathcal{A})$ be a measurable space, let $P$ and $Q$ 
be probability measures defined on $(\Omega, \mathcal{A})$
such that $P << Q$, let $X:\Omega\to\mathbb{R}$ be integrable both with respect to $P$ and $Q$,
and let the real valued function
 $\phi$ be convex in some (not necessarily bounded) interval 
containing the range of $X$.
Then
\begin{equation}\label{refJenRHS}
\int_{\Omega}  \phi (X) d P -  \phi \left(\int_{\Omega} X d P\right)
\le
\left\|\frac{d P}{d Q}\right\|_\infty
\left(\int_{\Omega}  \phi (X) d Q -  \phi \left(\int_{\Omega} X d Q\right)\right).
\end{equation}

\noindent Regarding the equality conditions, to avoid trivialities
we suppose that $P\ne Q$, and then we distinguish three cases. 

1)  Both sides
of (\ref{refJenRHS}) take the value 
$\infty$ if and only if  $\int_{\Omega}  \phi (X) d P = \infty$,
and then   either $\int_{\Omega}  \phi (X) d Q = \infty$ 
or $\left\|\frac{dP}{d Q}\right\|_\infty = \infty$. 

\noindent Next,  
assume that  $\phi$ is strictly convex, and
let $A:= \left\{\omega\in \Omega: \frac{d P}{d Q}(\omega) = 
\left\|\frac{d P}{d Q}\right\|_\infty\right\}$.

2) Both sides
of (\ref{refJenRHS}) take the value 
$0$ if and only if  $X$ is $Q$-a.e. constant. 

3) Both sides 
of (\ref{refJenRHS}) take the same  value $a$, with $0 < a < \infty$,
if and only if  the following three conditions hold:

a) $\int_{\Omega}  \phi (X) d Q < \infty$ and $\left\|\frac{d P}{d Q}\right\|_\infty < \infty$.

b) There exists a constant $c$ such that $X \equiv c$, $Q$-a.e.
on $\Omega\setminus A$, but $X$ is not $Q$-a.e. constant on $\Omega$.

c)  $Q (A) > 0$, $P (A) > 0$, and
$c = \int_{\Omega} X d Q = \frac{1}{ Q (A)} \int_{A} X d Q
 =\int_{\Omega} X d P = \frac{1}{ P (A)} \int_{A} X d P$.
\end{theorem}

\begin{example}\label{condition}  Given $(\Omega, \mathcal{A}, P)$, 
$A\subset X$ with $P(A)> 0$, and
 $X:\Omega\to\mathbb{R}$ integrable, 
it is intuitively obvious (and not difficult to prove directly)
that the variance $\operatorname{Var}_A (X)$ of $X$ restricted to $A$, with respect to
the conditional probability $P_A(B) := P(B|A)$, cannot be
much larger than the original variance of $X$ on all $\Omega$,
with respect to $P$.
This is now a special case of the preceding theorem: 
Since $\left\|\frac{d P_A}{dP}\right\|_\infty = \frac{1}{P(A)}$,
setting $\phi (t) = t^2$ in (\ref{refJenRHS}) we get 
$$\operatorname{Var}_A (X) \le
\frac{\operatorname{Var} (X)}{P(A)}.
$$
\end{example}

\begin{example}\label{condition} 
Given $X:\Omega\to\mathbb{R}$ and  $g: \mathbb{R}\to\mathbb{R}$,
suppose we know
the distribution of
$X$, but there is some uncertainty about the
value of one or several parameters, and we want to estimate how the variance
of $g(X)$ is affected by this
uncertainty. The preceding theorem can be used to this end.

Assume, for instance, that  $X \sim N(0, \sigma)$,
with $0 < a \le \sigma \le b$. The zero mean assumption is
made so the resulting
expressions are simple, but the case where there is
uncertainty both about the mean and the variance
can be treated in the same way. Call $\operatorname{Var}_1 (g(X))$
and $\operatorname{Var}_2 (g(X))$ the variances obtained by
supposing that $X$ has standard deviations $\sigma_1$ and
$\sigma_2$ respectively, with $a\le \sigma_1 \le \sigma_2\le b$, and let
 $P$ and $Q$ be the corresponding
laws for $X$. Then 
$$
\frac{d P}{dQ} = \frac{\sigma_2}{\sigma_1} 
e^{\frac{x^2(\sigma_1 - \sigma_2)}{2\sigma_1^2\sigma_2^2}},
\mbox{ \ \ \ so \ \ \ }
\left\|\frac{d P}{dQ}\right\|_\infty = \frac{\sigma_2}{\sigma_1} \le \frac{b}{a},
$$
 and hence
$$
\operatorname{Var}_1 (g(X)) \le \frac{b}{a} \operatorname{Var}_2 (g(X)).
$$
If we additionally know that $g$ is compactly supported,
we can present a simultaneous lower bound (in this
regard, see also Corollary  \ref{DraJen2} below).
Suppose, for simplicity, that
the support of $g$ is contained in $[c,d]$, with $0 < c < d$.
Then 
$$
\frac{d Q}{dP} = \frac{\sigma_1}{\sigma_2} 
e^{\frac{x^2(\sigma_2 - \sigma_1)}{2\sigma_1^2\sigma_2^2}} 
\le \frac{\sigma_1}{\sigma_2}  e^{\frac{d^2(\sigma_2 - \sigma_1)}{2\sigma_1^2\sigma_2^2}},
\mbox{ \ \ \ so \ \ \ }
\operatorname{Var}_2 (g(X)) 
\le \frac{\sigma_1}{\sigma_2}  e^{\frac{d^2(\sigma_2 - \sigma_1)}{2\sigma_1^2\sigma_2^2}} \operatorname{Var}_1 (g(X)),
$$
 and hence
$$
\frac{\sigma_2}{\sigma_1} e^{\frac{d^2(\sigma_1 - \sigma_2)}{2\sigma_1^2\sigma_2^2}} 
\operatorname{Var}_2 (g(X)) 
\le \operatorname{Var}_1 (g(X)).
$$
\end{example}

\begin{remark} It is natural to ask whether the hypothesis
$P << Q$ in the theorem, must be imposed on all of $\Omega$, or it is
sufficient to  consider just the set  $\{X\ne 0\}$. To see that this is
not enough, let $\Omega=[0,1]$, let $\mathcal{A}$ be the Lebesgue sets, and
let $d P = dx$. Take $X= \chi_{[1/2, 1]}$, and $dQ := 2 \chi_{[1/2, 1]} dx$.
Then $\frac{d P}{d Q} = \frac{1}{2}$ on $\{X\ne 0\}$, so
restricted to $\{X\ne 0\}$,
$P << Q$
and $\frac{dP}{dQ}$ is 
bounded. Let $\phi(x) = x^2$. Since $X$ is
constant a.e. with respect to $Q$, the right hand side of 
(\ref{refJenRHS}) is zero, while the left hand side is just $1/2 - 1/4$.

Even if we extend $\frac{d P}{d Q}$ to all $[0,1]$ by setting
$\frac{d P}{d Q} = \infty$ on $(0,1/2)$, the $\infty \times 0 = 0$
convention entails that (\ref{refJenRHS}) does not extend to
pairs $P, Q$ when $P$ has a singular part.
\end{remark}

\begin{remark}
The equality conditions, for the original Dragomir's result,  
where established in \cite{Al}.  In the more general measure-theoretic setting, dealing only with bounded functions would be
too restrictive, so the possibility that some term equals
$\infty$ must be considered. Related to this issue is the fact that,
since we are working on a probability space, the $L^p$ norm of any function is monotone
increasing in $p$. It is thus natural to ask whether the Dragomir-Jensen inequality can
be improved by replacing $\left\|\frac{dP}{dQ}\right\|_\infty$
with $c_p \left\|\frac{dP}{dQ}\right\|_p$, for some $p < \infty$ and some 
constant $c_p > 0$. We
show that $p=\infty$ is optimal.
Fix $p < \infty$ with  $p\ge 1$. Let $Q$ be Lebesgue measure on $(0,1]$ and let $\phi(x) = x^2$. 
Set 
$X(t) = t^{-1/2 + 1/(4 p)}$ and $dP(t) :=  C t^{-1/(2p)} dt $, where 
the constant $C$ is chosen so
$C t^{-1/(2p)}$ is a density. Then 
$X\in L^1(P)\cap L^1(Q)$, $\int_{\Omega}  \phi (X) d Q  < \infty$, 
$\int_{\Omega}  \phi (X) d P =  \infty$, and  $c_p \left\|\frac{dP}{dQ}\right\|_p < \infty$,
so 
\begin{equation*}
\int_{\Omega}  \phi (X) d P -  \phi \left(\int_{\Omega} X d P\right)
= \infty >
c_p \left\|\frac{dP}{dQ}\right\|_p
\left(\int_{\Omega}  \phi (X) d Q -  \phi \left(\int_{\Omega} X d Q\right)\right).
\end{equation*}
On the other hand, $\left\|\frac{dP}{dQ}\right\|_\infty = \infty$
and $0 < \int_{\Omega}  \phi (X) d Q -  \phi \left(\int_{\Omega} X d Q\right)$, so this is an instance where both sides of the Dragomir-Jensen
inequality take the value $\infty$.
\end{remark}

To make the structure of the proof of the theorem more transparent,
we place some measure-theoretic details into two technical lemmas.
 
\begin{lemma}\label{lemma1} Let $(\Omega, \mathcal{A})$ be a measurable space, let $P$ and $Q$ 
be  probability measures defined on $(\Omega, \mathcal{A})$
such that $P <<Q$, 
let $\frac{dP}{dQ}$ be essentially bounded, and let 
$$
A:= \left\{\omega\in \Omega: \frac{d P}{d Q}(\omega) = 
\left\|\frac{dP}{dQ}\right\|_\infty\right\}.
$$
 Then for every
measurable set $B\subset  A$, $P (B) = 0$ if and only if
$Q (B) = 0$. In particular, $P (A) = 0$ if and only if
$Q (A) = 0$.
\end{lemma}

\begin{proof} Since both $P$ and $Q$ are probabilities, 
$\left\|\frac{dP}{dQ}\right\|_\infty \ge 1$. But restricted to
 $A$,  $P$ is just    $\left\|\frac{dP}{dQ}\right\|_\infty Q$, so the result
 follows.
\end{proof}

\begin{lemma}\label{lemma2} With the notation and under the assumptions 
of the preceding
lemma, let the measurable 
space
$(\tilde{\Omega}, \tilde{\mathcal{A}})$ be the 
enlargement of  $(\Omega, \mathcal{A})$ obtained by adding a new point
$y$ to $\Omega$ and declaring it to be measurable. That is, $y\notin \Omega$,
$\tilde \Omega:= \Omega\cup \{y\}$, and 
 $\tilde{\mathcal{A}}:= \mathcal{A}\cup \{B\cup \{y\}: B\in \mathcal{A}\}$.
 Denote by $\delta_y$ the Dirac point mass at $y$.
Then the set function $\tilde{P}$, defined on $(\tilde \Omega, \tilde{\mathcal{A}})$ by 
\begin{equation}\label{tau}
\tilde{P} (B):= Q (B\cap \Omega) - \left\|\frac{dP}{dQ}\right\|_\infty^{-1} P(B\cap \Omega) + \left\|\frac{dP}{dQ}\right\|_\infty^{-1}  \delta_y (B),
\end{equation}
is a probability. Furthermore,  $\tilde{P}(A) = 0$, and for every
measurable set $B\subset  \Omega\setminus A$, $\tilde{P} (B) = 0$ if and only if
$Q (B) = 0$. 
\end{lemma}

\begin{proof} 
Select any $w\in \Omega$, identify $\Omega$ with
$\Omega\times \{0\}$, and set $y:= w\times \{1\}$ (this ensures
that $y$ is actually a new point). Then let $\tilde \Omega:= \Omega\cup \{y\}$, and let
  $\tilde{\mathcal{A}}:= \mathcal{A}\cup \{B\cup \{y\}: B\in \mathcal{A}\}$. Since $\tilde{\mathcal{A}}$ contains $\emptyset$, is
closed under complementation, and also under countable unions,  it is
a 
$\sigma$-algebra. Additionally, $\{y\}$ is measurable, for
$ \emptyset  \in  \mathcal{A}$, whence
$\{y\} = \emptyset \cup \{y\}\in \tilde{\mathcal{A}}$.

Let us check that $\tilde{P}$ is a probability. Clearly $\tilde{P}(\tilde \Omega) = 1$; to see that $\tilde{P}$ is non-negative, we use the
change of variable formula 
$
\int_{\Omega}   X d P = \int_{\Omega}   X \frac{d P}{d Q} d Q $, and we conclude that
 for every $B\in \tilde{\mathcal{A}}$, we have 
$$
\tilde{P}(B) \ge \int_{B\cap \Omega} \left(1 - \left\|\frac{dP}{dQ}\right\|_\infty^{-1}\frac{d P}{d Q} \right)d Q \ge 0.
 $$

Note next that whenever $B\subset \Omega$ is measurable, (\ref{tau}) reduces to 
\begin{equation}\label{tau1}
\tilde{P} (B) =  Q (B) - \left\|\frac{dP}{dQ}\right\|_\infty^{-1} P(B) \le Q (B).
\end{equation}
In particular, $A:= \left\{w\in \Omega: \frac{d P}{d Q}(w) = 
\left\|\frac{dP}{dQ}\right\|_\infty\right\}$ has $\tilde{P}$-measure
zero, since  by (\ref{tau1}),
$$
\tilde{P} (A):= Q (A) - \left\|\frac{dP}{dQ}\right\|_\infty^{-1} P (A)
 = Q (A) - \int_A \left\|\frac{dP}{dQ}\right\|_\infty^{-1} \frac{dP}{dQ} dQ = 
 Q (A) - Q (A) 
 = 0.
$$ 
Finally, to see that for any measurable $B\subset \Omega\setminus A$, 
$\tilde{P} (B) = 0$ if and only if $Q (B) = 0$,
 observe  that if $Q (B) = 0$, by (\ref{tau1}) we have
$\tilde{P} (B) \le Q (B) = 0$, while if $Q (B) > 0$, 
since $1 - \left\|\frac{dP}{dQ}\right\|_\infty^{-1}\frac{d P}{d Q} > 0$ on $\Omega\setminus A$, we have
$$
\tilde{P}(B) 
= \int_{B} \left(1 - \left\|\frac{dP}{dQ}\right\|_\infty^{-1}\frac{d P}{d Q} \right)d Q > 0.
 $$
\end{proof}

\begin{proof}[Proof of the Theorem] Note that if $\frac{dP}{dQ}$ is unbounded (i.e., if $\left\|\frac{dP}{dQ}\right\|_\infty = \infty$) then the right hand side of
(\ref{refJenRHS}) is $\infty$ unless 
$\int_{\Omega}  \phi (X) d Q =  \phi \left(\int_{\Omega} X d Q\right)$,
in which case its value is zero (by the standard convention). 
So for (\ref{refJenRHS}) to hold we need to have 
$\int_{\Omega}  \phi (X) d P =  \phi \left(\int_{\Omega} X d P\right)$.
This would be trivial if $\phi$ were strictly convex, for then
$X$ would be constant $Q$-a.e., and thus also constant $P$-a.e.
since $P <<Q$.
But for $\phi$ not strictly convex, equality may occur for non-constant
functions, and so a different argument is needed. What we do is
to  assume first that 
$
\left\|\frac{dP}{dQ}\right\|_\infty < \infty$, and then handle the unbounded
case via an approximation argument.
Let us suppose  also that both  $\int_{\Omega}  \phi (X) d Q 
 < \infty$ and $\int_{\Omega}  \phi (X) d P  < \infty$. Then inequality (\ref{refJenRHS}) is   equivalent
to
\begin{equation}\label{step1} 
\phi \left(\int_{\Omega} X d Q\right)
\le \int_{\Omega}  \phi (X) d Q
 -  \left\|\frac{dP}{dQ}\right\|_\infty^{-1}
\int_{\Omega} \phi (X) d P + \left\|\frac{dP}{dQ}\right\|_\infty^{-1} 
\phi \left(\int_{\Omega} X d P\right),
\end{equation}
simply by rearranging terms. 
But (\ref{step1}) immediately follows from the usual Jensen's inequality, 
applied to the probability measure  space $(\tilde{\Omega}, \tilde{\mathcal{A}}, \tilde{P})$ defined in Lemma \ref{lemma2}, and to a suitable   extension of $X$, given by  
$\tilde{X}:= X$ on $\Omega$, and
$\tilde{X}(y):= \int_{\Omega} X d P$. The function $\tilde{X}:\tilde{\Omega}\to \mathbb{R}$ is clearly $\tilde{\mathcal{A}}$ measurable. Furthermore,
\begin{equation}\label{eqint}
\int_{\tilde{\Omega}} \tilde X d \tilde{P} 
= \int_{\Omega} \tilde X d \tilde{P} + \int_{\{y\}} \tilde X d \tilde{P}
= \int_{\Omega}  X d Q - 
\left\|\frac{dP}{dQ}\right\|_\infty^{-1} \int_{\Omega}  X d {P} 
 +  \left\|\frac{dP}{dQ}\right\|_\infty^{-1} \int_{\{y\}} \left(\int_{\Omega}   X d P\right)  d \delta_y 
\end{equation}
\begin{equation}\label{eqint1} 
= \int_{\Omega}  X d Q,
\end{equation}
so by  Jensen's inequality on $(\tilde{\Omega}, \tilde{\mathcal{A}}, \tilde{P})$, 
\begin{equation}\label{eqint2}  
\phi \left(\int_{\Omega} X d Q\right)
= 
\phi \left(\int_{\tilde{\Omega}} \tilde X d \tilde{P} \right) \le \int_{\tilde{\Omega}} \phi \left( \tilde X \right) d \tilde{P}
\end{equation}
\begin{equation}\label{eqint3}  
=
\int_{\Omega}  \phi (X) d Q
 -  \left\|\frac{dP}{dQ}\right\|_\infty^{-1}
 \int_{\Omega} \phi (X) d P +  
\left\|\frac{dP}{dQ}\right\|_\infty^{-1}
\int_{\{y\}} \phi \left( \tilde X \right)  d \delta_y.
\end{equation}
It follows from  
$\tilde X(y) = \int_{\Omega} X d P$ that
$\int_{\{y\}} \phi \left( \tilde X \right) d \delta_y = \phi \left( \tilde X (y) \right)  =
\phi\left(\int_{\Omega} X d P\right)$,  so (\ref{step1}) holds,
 and hence so does (\ref{refJenRHS}) when every term appearing there is finite. Suppose next that 
$\left\|\frac{dP}{dQ}\right\|_\infty
= \infty$, and that 
$\int_{\Omega}  \phi (X) d Q =  \phi \left(\int_{\Omega} X d Q\right)$ (for otherwise (\ref{refJenRHS}) is trivial). Let us emphasize that we
do not a priori assume that $\int_{\Omega}  \phi (X) d P < \infty$;
this will follow once
 $\int_{\Omega} \phi (X) d P =  \phi \left(\int_{\Omega} X d P\right)$
 is proven. 
Note however that $\int_{\Omega} (\phi(X))^- dP 
< \infty$, since
$\int_{\Omega} \phi(X) dP \ge
\phi\left(\int_\Omega  X dP \right) > -\infty$.
Define $d P_n := \min\left\{\frac{dP}{dQ}, n\right\} d Q$, 
and observe  that by the monotone convergence theorem, applied 
to $\min\left\{\frac{dP}{dQ}, n\right\}$, and separately to  
$X^+ \min\left\{\frac{dP}{dQ}, n\right\}$, $X^-\min\left\{\frac{dP}{dQ}, n\right\}$,  
$(\phi(X))^+ \min\left\{\frac{dP}{dQ}, n\right\}$,  and $(\phi(X))^-\min\left\{\frac{dP}{dQ}, n\right\} $,  we have
$$
\lim_n P_n (\Omega) = \lim_n \int_\Omega \min\left\{\frac{dP}{dQ}, n\right\} d Q = \int_\Omega \frac{dP}{dQ} d Q =
P(\Omega) = 1,
$$ 
$$
\lim_n \frac{1}{P_n (\Omega)} \int_\Omega X\ \min\left\{\frac{dP}{dQ}, n\right\} d Q
 = 
 \int_\Omega X  \ \frac{dP}{dQ} d Q 
=
\int_\Omega X d P,
$$
and
$$
\lim_n \frac{1}{P_n (\Omega)} \int_\Omega \phi(X) dP_n = 
\int_\Omega \phi(X) d P.
$$
Since $\left\|\frac{dP_n}{dQ}\right\|_\infty
= n$, we know from Jensen's inequality and the previous case
that
$$
0\le \int_\Omega \phi(X)\ \frac{dP_n}{P_n (\Omega)} - 
\phi\left(\int_\Omega  X\ \frac{dP_n}{P_n (\Omega)}\right)
\le n \left[\int_\Omega \phi(X) dQ - 
\phi\left(\int_\Omega X   dQ\right)\right]
=0.
$$
Thus, for every $n$ we have 
\begin{equation}\label{equality1}
\int_\Omega \phi(X)\ \frac{dP_n}{P_n (\Omega)} =
\phi\left(\int_\Omega  X\ \frac{dP_n}{P_n (\Omega)}\right).
\end{equation}
Since all the limits involved exist, letting $n\to\infty$
and using the continuity of $\phi$ in the interior of its domain  $I$, we conclude that
\begin{equation}\label{equality2}
\int_\Omega \phi(X) dP =
\phi\left(\int_\Omega  X dP \right),
\end{equation}
(this is stated below as a Corollary, as it seems to be of
independent interest). 
Note that the interval 
$I$ might contain one  (or both) of its endpoints, say it contains the endpoint $a$, and 
$\phi$ might be discontinuous there. But if $\int_\Omega  X dP = a$,
since the range of $X$ is contained in $I$, we have
$X =a$ $P$-a.e., and thus (\ref{equality2}) also holds in this case.

So far, we know that if 
$\int_{\Omega} \phi (X) dP < \infty$ and $\int_{\Omega} \phi (X) dQ < \infty$,
then
inequality (\ref{refJenRHS}) holds, regardless of whether or not
$\frac{dP}{dQ}$ is essentially bounded, and additionally, that if 
$
\int_\Omega \phi(X) dQ =
\phi\left(\int_\Omega  X dQ \right)
$, then
$
\int_\Omega \phi(X) dP =
\phi\left(\int_\Omega  X dP \right)
$. Furthermore, it is clear that if 
$\int_{\Omega} \phi (X) dP = \infty$ and
$\left\|\frac{dP}{dQ}\right\|_\infty <\infty$, we must have
$\int_{\Omega} \phi (X) dQ = \infty$, so all we need to show
is that $\left\|\frac{dP}{dQ}\right\|_\infty =\infty$ 
whenever $\int_{\Omega} \phi (X) dP = \infty$ and 
$\int_{\Omega} \phi (X) dQ < \infty$.
This latter inequality, together with $\int_{\Omega} \phi (X) dQ \ge
\phi\left(\int_{\Omega}  X dQ \right) > -\infty$, entail that both
the positive and negative parts of $\phi (X)$ are $Q$-integrable.
Suppose, towards a contradiction, that 
$\left\|\frac{dP}{dQ}\right\|_\infty <\infty$.
Then 
$$
\int_{\Omega} \phi (X) dP \le \int_{\Omega} \left|\phi (X)  \ \frac{dP}{dQ} \right|dQ  \le \left\|\frac{dP}{dQ}\right\|_\infty \int_{\Omega} \left|\phi (X)\right| dQ  < \infty.
$$

In view of the fact that inequality (\ref{refJenRHS}) holds under no restrictions, the
first equality case, where both sides take the value $\infty$, follows immediately.

Next, suppose that $\phi$ is strictly convex. By the equality case of Jensen's inequality the right hand side of (\ref{refJenRHS})
is zero (hence, so is the left hand side) if and only if $X$ is constant $Q$-a.e.; thus, statement 2) 
of the theorem is true.

Finally, suppose there exists a constant $a\in(0,\infty)$ such that
\begin{equation}\label{equality3}
a = \int_{\Omega}  \phi (X) d P -  \phi \left(\int_{\Omega} X d P\right)
=
\left\|\frac{dP}{dQ}\right\|_\infty
\left(\int_{\Omega}  \phi (X) d Q -  \phi \left(\int_{\Omega} X d Q\right)\right).
\end{equation}
Then a) is immediate; using (\ref{refJenRHS})  we conclude that $\int_{\Omega}  \phi (X) d P$
is also finite, so  equality on (\ref{refJenRHS}) is equivalent to having
equality on (\ref{step1}), which in turn is equivalent to $\phi \left(\int_{\tilde{\Omega}} \tilde X d \tilde{P} \right) = \int_{\tilde{\Omega}} \phi \left( \tilde X \right) d \tilde{P}$. Now $X$ is not constant 
$Q$-a.e. on $\Omega$ (by the equality
case in Jensen's inequality, otherwise $a$ would be zero)
but $\tilde X$ {\em is} constant $\tilde{P}$-a.e. 
by the equality
case in Jensen's inequality, call this constant c. Since by Lemma \ref{lemma2}, $\tilde{P}$ and $Q$ 
are mutually absolutely continuous
over the set $\Omega\setminus A$, statement b) follows.

To see why c) holds, note first
that if $Q(A) = 0$, then  by b) $X\equiv c$ $Q$-a.e. and  we are in the
case $0=0$. But $a > 0$, so $Q(A) > 0$, and hence $P(A) > 0$
by Lemma \ref{lemma1}.   Since $P \ne Q$, 
we also have $\tilde{P} (\Omega\setminus A) > 0$. Using $\tilde X\equiv c$  $\tilde{P}$-a.e.
together with 
(\ref{eqint})-(\ref{eqint1}),
we see that for  $\tilde{P}$-a.e. (or $Q$-a.e.) $w \in \Omega\setminus A$,
\begin{equation}\label{ave0}
\int_{\Omega} X d Q = 
\int_{\tilde{\Omega}} \tilde X d \tilde{P} 
= c
 = X(w) = \tilde{X}(y) = \int_{\Omega} X d P.
\end{equation}
Since $X=c$ a.e. (with respect to all measures under
consideration) on $\Omega\setminus A$,
$$
\int_{\Omega\setminus A} X d Q + \int_{A} X d Q = c (1 - Q (A)) + \int_{A} X d Q
= \int_{\Omega} X d Q = c,
$$
so
\begin{equation}\label{ave1}
c =\frac{1}{Q (A)} \int_{A} X d Q.
\end{equation}
The fact that 
\begin{equation}\label{ave2}
c =\frac{1}{P (A)} \int_{A} X d P
\end{equation}
is obtained in the same way. Alternatively, since
$
\int_{\Omega} X d Q = \int_{\Omega} X d P
$ and
$X= c$ a.e. on $X\setminus A$, we see that (\ref{ave1}) holds if and only if
(\ref{ave2}) does.

Suppose now that a), b) and c) hold. Using a) and b) we conclude that the
right hand side of (\ref{refJenRHS}) is neither zero nor infinity. Since all terms involved are finite,
equality on (\ref{refJenRHS}) is equivalent to $\phi \left(\int_{\tilde{\Omega}} \tilde X d \tilde{P} \right) = \int_{\tilde{\Omega}} \phi \left( \tilde X \right) d \tilde{P}$, which holds if $\tilde X$ is
$\tilde{P}$-a.e. constant on $\tilde{\Omega}$. We prove this next.
First,  $\tilde X\equiv c$ $\tilde{P}$-a.e. on $\Omega\setminus A$ by
b) and Lemma \ref{lemma2}, while by definition and by c), 
${\tilde X}(y) =  \int_{\Omega} X d P
= c$. Since $\tilde{P} (A) = 0$ by Lemma \ref{lemma2}, the result follows.
\end{proof}

The next corollary was obtained as a step in the preceding proof, and of
course, it also  follows directly from the statement of Theorem \ref{DraJen}.

\begin{corollary} Let $(\Omega, \mathcal{A}, Q)$ be a probability space, let $X:\Omega\to\mathbb{R}$ be integrable,
and let $\phi$ be real-valued and convex in some interval 
containing the range of $f$. If $\phi \left(\int_{ \Omega} X d Q \right) = \int_{ \Omega} \phi (X) d Q$,
then for every probablity $P <<Q$ such that $X\in L^1(P)$, we have $\phi \left(\int_{ \Omega}  X d P \right) = \int_{ \Omega} \phi (X) d P$.
\end{corollary}

\begin{proof} By (\ref{refJenRHS}), since $\int_{\Omega}  \phi (X) d P -  \phi \left(\int_{\Omega} X d P\right) \ge 0$.
\end{proof}

We state next the version of Theorem \ref{DraJen} (where for simplicity we omit
the analogous equality conditions)  that provides a lower bound
(instead of an upper bound)  using the essential infimum of the Radon-Nykodim derivative  (instead
of the essential supremum). Recall that if $P$ and $Q$ are mutually
absolutely continuous and  $h$ is any representative of $\frac{d P}{d Q}$,
then $1/h$ is a representative of $\frac{d Q}{d P}$. 

\begin{corollary}\label{DraJen2} Let $(\Omega, \mathcal{A})$ be a measurable space, let $P$ and $Q$ 
be  probability measures defined on $(\Omega, \mathcal{A})$
such that $P <<Q$, let $X:\Omega\to\mathbb{R}$ be integrable both with respect to $P$ and $Q$,
and let the real valued function
 $\phi$ be convex in some  interval 
containing the range of $X$.
Then
\begin{equation}\label{refJenLHS} 
\left(\operatorname{ess\  inf} \frac{d P}{d Q}\right)\left(\int_{\Omega}  \phi (X) d Q -  \phi \left(\int_{\Omega} X d Q\right)\right)
\le \int_{\Omega}  \phi (X) d P -  \phi \left(\int_{\Omega} X d P\right).
\end{equation}
\end{corollary}

\begin{proof} If  $
\operatorname{ess\  inf} \frac{d P}{d Q}
= 0$, then  inequality (\ref{refJenLHS})  reduces to the usual Jensen inequality,
so only when 
$
\operatorname{ess\  inf} \frac{d P}{d Q}
> 0$ does (\ref{refJenLHS}) say something new. But in this case
we have $Q <<P$, with $\frac{d Q}{d P} = \frac{1}{d P/d Q}$ a bounded function, since 
$\left\|\frac{dQ}{dP}\right\|_\infty = \frac{1}{\operatorname{ess\ inf} \frac{d P}{d Q}}< \infty$. Now inequality  (\ref{refJenLHS}) follows from (\ref{refJenRHS}):  Multiply both sides of (\ref{refJenLHS})
by $\left\|\frac{dQ}{dP}\right\|_\infty$, and 
note that this is just (\ref{refJenRHS}) with the roles of $P$ and $Q$ interchanged.
\end{proof}

Since in the nontrivial case 
$
\operatorname{ess\  inf} \frac{d P}{d Q}
> 0$
the preceding corollary reduces to Theorem \ref{DraJen}, the
corresponding equality conditions follow automatically.

If $\phi$ is  concave, then applying Theorem \ref{DraJen}
and Corollary \ref{DraJen2} to $-\phi$ we obtain the following

\begin{corollary}\label{DraJen3} Let $(\Omega, \mathcal{A})$ be a measurable space, let $P$ and $Q$ 
be  probability measures defined on $(\Omega, \mathcal{A})$
such that $P <<Q$, let $X:\Omega\to\mathbb{R}$ be integrable both with respect to $P$ and $Q$,
and let the real valued function
 $\phi$ be concave in some  interval 
containing the range of $X$.
Then
\begin{equation}  
\left(\operatorname{ess\  inf} \frac{d P}{d Q}\right)\left(
\phi \left(\int_{\Omega} X d Q\right) - \int_{\Omega}  \phi (X) d Q\right)
\le 
\end{equation}
\begin{equation}
\phi \left(\int_{\Omega} X d P\right)
 - \int_{\Omega}  \phi (X) d P  
\le
\left\|\frac{dP}{dQ}\right\|_\infty
\left(\phi \left(\int_{\Omega} X d Q\right) - \int_{\Omega}  \phi (X) 
d Q\right).
\end{equation}
\end{corollary}

Next we state the corresponding refinement of the
measure-theoretic version of the arithme\-tic-geometric mean inequality
$
\exp\int_{\Omega} \log( X)  dP \le \int_{\Omega} X  dP,
$
thereby generalizing \cite[Theorem 2.1]{Al}.

\begin{corollary}\label{DraJen4} Let $(\Omega, \mathcal{A})$ be a measurable space, let $P << Q$ 
be   probability measures defined on $(\Omega, \mathcal{A})$, and let 
$X:\Omega\to [0,\infty) $ be such that $\log \Omega$ is integrable both with respect to $P$ and $Q$. 
Then
\begin{equation} 
\left(\operatorname{ess\  inf} \frac{d P}{d Q}\right)\left(
\int_{\Omega} X  dQ - \exp\int_{\Omega} \log( X)  d Q\right)
\le 
\end{equation}
\begin{equation}\label{refJenRHS2}
\int_{\Omega} X  dP - \exp\int_{\Omega} \log( X)  dP
\le
\left\|\frac{dP}{dQ}\right\|_\infty
\left(\int_{\Omega} X  dQ - \exp\int_{\Omega} \log( X)
d Q\right).
\end{equation}
\end{corollary}

Let us  finish by saying that 
 the reader interested in applications of
the Dragomir-Jensen inequality to information inequalities, 
can find some such applications (for the discrete case) in the original Dragomir's paper \cite{Dra}.

\end{document}